\documentclass[12pt]{article}
\usepackage{amssymb,amsmath,amsthm}

\newtheorem{thm}     {Theorem}[section]
\newtheorem{prop}    [thm]{Proposition}

\newtheorem{lemma}   [thm]{Lemma}
\newtheorem{remark}   [thm]{Remark}

\newcommand{\N}{\mathbb N}
\newcommand{\R}{\mathbb R}
\newcommand{\Z}{\mathbb Z}
\def\rank{{\rm rank}\,}
\def\<{\langle}
\def\>{\rangle}

\begin{document}

\title{{\bf Scarcity of periodic orbits in outer billiards}}
\author{Alexander Tumanov\footnote{
Partially supported by Simons Foundation grant.}}
\date{}
\maketitle
\rightline{\it To the memory of my dear teacher Gennadi Henkin}
\bigskip
\bigskip

University of Illinois, Department of Mathematics,
1409 West Green St., Urbana, IL 61801,
{\it tumanov@illinois.edu}
\bigskip

{\bf Abstract.}
We give a simple proof of the result of \cite{Tum-Zhar} that the set of period 4 orbits in planar outer billiard with piecewise smooth convex boundary has empty interior, provided that no four corners of the boundary form a parallelogram.
We also obtain results on period 5 and 6 orbits.

\bigskip

MSC: 37D50, 37C25

Key words: Outer billiard, periodic orbit, exterior differential system.
\bigskip

\section{Introduction}

In this paper we study the set of periodic orbits in
planar outer billiard.

Let $\Gamma$ be a piecewise smooth convex curve in the plane $\R^2$. The dynamics of the outer (or dual) billiard is defined in the exterior of $\Gamma$ as follows.
The outer billiard map sends a point $z_1$ to the point $z_2$, so that the line between the two points is a supporting line for $\Gamma$ and meets $\Gamma$ at the midpoint of the segment $[z_1,z_2]$. 

The outer billiard was originally introduced by Bernhard Neumann (1960-s) as a model for a stability problem.
Moser \cite{Moser} attracted much attention to the question
whether all orbits in an outer billiard are bounded. 
Schwartz \cite{Schwartz} did extensive work on the question
for polygonal outer billiards.
See \cite{Tabachnikov} for more discussion and related results.

A relevant question on the outer billiard
is whether the set of periodic points (orbits) must
have measure zero, or in a milder version -- empty interior.
The interest to this question comes from the classical
Birkhoff billiard, for which the question is related to
asymptotic distribution of large eigenvalues of the
Laplacian \cite{Ivrii}.

The question has turned out to be surprisingly hard
and remains open.
For the classical planar billiard, Rychlik \cite{Rychlik}
showed that the set of period 3 points has measure zero. 
The proof involved symbolic calculations that were
later removed by Stojanov \cite{Stojanov}. 
Wojtkovski \cite{Wojtkovski} gave a simpler proof relying on Jacobi's fields.
Vorobets \cite{Vorobets} extended this result to higher dimensional billiards. Glutsyuk and Kudryashov
\cite{Glutsyuk-Kudryashov} showed
that the set of period 4 points in a planar billiard
has measure zero.
For the planar outer billiard,
Genin and Tabachnikov \cite{Gen-Tab}
proved that the set of period 3 points has empty
interior. The author and Zharnitsky \cite{Tum-Zhar} proved
that the set of period 4 points has empty interior
unless there are four corners of $\Gamma$ that form a parallelogram. The proof in \cite{Tum-Zhar} involves
computer aided computations.

In this paper we give a simple proof of the result of \cite{Tum-Zhar} avoiding heavy computations.
In addition to the result about period 4,
we obtain results about period 5 and 6 orbits.

Following \cite{Tum-Zhar}, we use an approach
based on exterior differential systems (EDS).
This approach was introduced by Baryshnikov, Landsberg,
and Zharnitsky, see \cite{Bar-Zhar, Landsberg}.
An EDS on a smooth manifold $X$ is a subbundle $D$ of the
tangent bundle $T(X)$. We call a smooth manifold $M\subset X$ an integral manifold for $D$
if the tangent space of $M$ is contained in $D$.
There is an explicitly defined EDS $D$ in $\R^{2n}$
so that 2-dimensional integral manifolds of $D$ correspond
to outer billiards with open sets of period $n$ points
(see Proposition \ref{reduction-to-EDS}).
Thus, the question whether there is an outer billiard with
an open set of periodic points reduces to the question
whether there is a 2-dimensional integral manifold for $D$.

There is an algorithm by E. Cartan and K\"ahler
(see, e.g. \cite{Ivey}) that answers the question
whether an EDS has integral manifolds
of certain dimension, however, the algorithm often leads to
calculations that are difficult to accomplish
even on a computer.
The first step in the algorithm consists of describing
integral elements, which are subspaces of $D$ of given dimension (in our case, of dimension 2) that can
be candidates for tangent spaces to $M$.
They are subspaces on which the defining 1-forms
for $D$ vanish together with their exterior differentials.
In this paper we do not go beyond this first step.
Our main idea is based on an observation that the convexity
of the billiard curve $\Gamma$ can be formulated in terms
of integral elements. It turns out that in some cases
that we describe here, there are no integral elements that
arise from a convex curve, therefore there does not exist
an integral manifold that corresponds to a convex outer billiard.

We point out, however, that for the conventional, inner billiards, our approach does not work. That is, the convexity
of the billiard curve does not affect the existence of
integral elements for the corresponding EDS.

The paper was written for a special volume of JGA
dedicated to the memory of Gennadi Henkin.
I am deeply grateful to Gennadi for advising me as a student,
and for life-long support and inspiration.

\section{Outer billiard and EDS}

Let $G\subset\R^2$ be a convex bounded domain with piecewise
$C^2$ boundary $\Gamma$. We recall the definition of the
outer billiard map $F$. Let $z_1\in \R^2\setminus\bar G$. We put $F(z_1)=z_2$ if the oriented line $L$ from $z_1$ to
$z_2$ is a supporting line for $G$ that meets $\Gamma$ at the midpoint $(z_1+z_2)/2$ and for definiteness $G$ lies on the left of $L$. The line $L$ is either tangent to $\Gamma$ or passes through the corner.

The map $F$ is well defined in the exterior of $G$ except
on supporting lines $L$ for which the intersection
$\Gamma\cap L$ is not a single point. These lines are countable, and their union $E$ is nowhere dense.
The map $F$ is $C^1$ smooth at $z_1$ unless $\Gamma$ has zero
curvature at the midpoint $(z_1+z_2)/2$ or $L$ is one of the
two tangent lines at a corner.
The union $E_0$ of tangent lines $L$ through points of zero curvature is a closed nowhere dense set, and $E\subset E_0$.

The orbit $O(z)$ of the point $z$ is the set
$O(z)=\{F^j(z): j\in\N\}$, here $F^j$ is the $j$-th iterate
of $F$. We call $z$ a period $n$ point if $F^n(z)=z$.
In this case we also call $O(z)$ a period $n$ orbit.

We are concerned with the question whether the set of periodic points can have nonempty interior. Following \cite{Tum-Zhar} we reduce the problem to the existence of integral manifolds of a certain exterior differential system (EDS), which we now introduce.

We view the set of period $n$ orbits as a subset in
\[
\R^{2n}=\{(z_1,\ldots,z_n) :
z_i=(x_i,y_i)\in\R^2, i\in\Z/n\Z \}.
\]
Define the differential forms
\begin{equation}\label{theta}
\theta_i=\det(d\bar r_i,r_i), \qquad
r_i=\frac{z_i-z_{i+1}}{2}, \qquad
\bar r_i=\frac{z_i+z_{i+1}}{2}, \qquad
i\in\Z/n\Z.
\end{equation}
Here $\det$ stands for the determinant.
We will see (Proposition \ref{coframe}) that if no three consecutive points
$z_i, z_{i+1}, z_{i+2}$ are collinear, then the forms
$\theta_i$ are linearly independent, and their common zero
set forms a rank $n$ subbundle $D$ of the tangent bundle of $\R^{2n}$, an EDS.
We call a smooth manifold $M\subset\R^{2n}$
an {\it integral manifold} for $D$
if the tangent space of $M$ is contained in $D$, that is,
$\theta_i|_M=0$ for all $i$.

\begin{prop}{\rm \cite{Tum-Zhar}}\label{reduction-to-EDS}
Let the set of period $n$ points have non-empty interior.
Then there exists $M$, an integral manifold for $D$ such that
$\dim M=2$ and $dx_1\wedge dy_1|_M$ is non-vanishing.
\end{prop}
The converse is also true, but we do not need it here.
For completeness, we give a proof.
\begin{proof}
Let the set of period $n$ points have non-empty interior.
Then there exists an open set $U\subset\R^2$ such that
all iterates $F^j$, $1\le j\le n-1$, are $C^1$ smooth in $U$, and for every $z_1\in U$, we have $F^n(z_1)=z_1$. Define
\begin{equation}\label{integral-manifold}
M=\{(z_1,\dots,z_n):z_1\in U, z_{i+1}=F(z_i), i\in\Z/n\Z \}.
\end{equation}
Clearly $\dim M=2$ and $dx_1\wedge dy_1|_M$ is non-vanishing.
We claim that $M$ is an integral manifold for $D$.
Let $z(t)=(z_1(t),\ldots,z_n(t))$ be a curve in $M$.
Then $\bar r_i(t)\in\Gamma$. If for some $t$, the midpoint
$\bar r_i(t)$ is a corner, then $\bar r_i'(t)=0$, hence
$\theta_i(z'(t))=0$. Otherwise, $\bar r_i'(t)$ is tangent
to $\Gamma$ at $\bar r_i(t)$. Then the vectors
$\bar r_i'(t)$ and $r_i(t)$ are collinear, and again
$\theta_i(z'(t))=0$ as desired.
\end{proof}

\section{Integral elements}

We say that a 2-subspace $\sigma\subset T_z \R^{2n}$,
$z\in\R^{2n}$, is an {\it integral element} for $D$ if
$\theta_i|_\sigma=0$ and $d\theta_i|_\sigma=0$ for all $i$.
Clearly, all tangent spaces to $M$ defined by \eqref{integral-manifold} are integral elements.
To describe integral elements, we include the forms $\theta_i$ in a coframe. Following \cite{Tum-Zhar}, we put
\[
\omega_i=\det(dr_i,r_i).
\]
\begin{prop}{\rm \cite{Tum-Zhar}}\label{coframe}
The forms $\theta_i, \omega_i$, $i\in\Z/n\Z$, are linearly
independent provided that no three consecutive points are
collinear.
\end{prop}
For completeness and subsequent use we include a proof.
\begin{proof}
It suffices to express the standard frame $dz_i$
in terms of the forms $\theta_i, \omega_i$.
We first recall how to solve the system
\[
\det(x,a)=\alpha,\qquad
\det(x,b)=\beta
\]
for given $a,b\in\R^2$ and $\alpha,\beta\in\R$. One verifies if $\Delta=\det(a,b)\ne0$, then the system has a unique solution
\begin{equation}\label{solving-system-with-det}
  x=\frac{\beta a-\alpha b}{\Delta}.
\end{equation}
We introduce
\[
\Delta_{ij}=\det(r_i,r_j),\qquad
\Delta_i=\Delta_{i-1, i}.
\]
Geometrically, $\Delta_i$ is the half of the area of the
triangle with vertices $z_{i-1}, z_i, z_{i+1}$.
These points are collinear exactly when $\Delta_i=0$.
We have
\[
z_i=\bar r_i+r_i,\qquad
z_{i+1}=\bar r_i-r_i.
\]
Accordingly, we get
\[
\det(dz_i,r_i)=\theta_i+\omega_i,\qquad
\det(dz_{i+1},r_i)=\theta_i-\omega_i.
\]
Solving the system
\[
\det(dz_i,r_{i-1})=\theta_{i-1}-\omega_{i-1},\qquad
\det(dz_i,r_i)=\theta_i+\omega_i
\]
using \eqref{solving-system-with-det} yields
\begin{equation}\label{dz-i}
dz_i=\frac{(\omega_i+\theta_i)r_{i-1}
+(\omega_{i-1}-\theta_{i-1})r_i}{\Delta_i},
\end{equation}
as desired.
\end{proof}

For a period $n$ orbit in outer billiard,
no three consecutive points $z_{i-1}, z_i, z_{i+1}$ are collinear, hence $\Delta_i>0$,
and $\theta_i, \omega_i, i\in\Z/n\Z$ form a coframe,
which we assume from now on.

\begin{prop}{\rm \cite{Tum-Zhar}}
\label{prop-on-integral-element}
Let $\sigma\subset T_z \R^{2n}$ be a 2-subspace such that
$\theta_i|_\sigma=0$ for all $i$. Then $\sigma$ is an integral element for $D$ if and only if for all $i,j\in\Z/n\Z$ we have
\begin{equation}\label{equation-on-integral-element}
 \Delta^{-1}_i\omega_{i-1}\wedge\omega_i|_\sigma
=\Delta^{-1}_j\omega_{j-1}\wedge\omega_j|_\sigma
\end{equation}
\end{prop}
For completeness we include a proof.
\begin{proof}
Differentiating \eqref{theta}, we get
\begin{equation}\label{d-theta-i}
d\theta_i=-\det(d\bar r_i,dr_i)
=-\frac{1}{4}(\det(dz_i,dz_i)-\det(dz_{i+1},dz_{i+1})),
\end{equation}
where the determinants are calculated using the wedge
product. Note that the determinant of two 1-forms is
commutative.
Since $\theta_i|_\sigma=0$, the equation \eqref{dz-i}
turns into
\begin{equation*}
dz_i|_\sigma=\Delta^{-1}_i(\omega_i r_{i-1}+\omega_{i-1}r_i).
\end{equation*}
Then we obtain
\begin{equation}\label{dzi-dzi}
\det(dz_i,dz_i)
=\Delta^{-2}_i \det(\omega_i r_{i-1}+\omega_{i-1}r_i,
\omega_i r_{i-1}+\omega_{i-1}r_i)
=-2 \Delta^{-1}_i \omega_{i-1}\wedge\omega_i.
\end{equation}
The equation \eqref{equation-on-integral-element}
now follows by \eqref{d-theta-i} and \eqref{dzi-dzi}.
\end{proof}

We now derive parametric equations of integral elements.
We introduce a $n\times n$ matrix
\begin{equation}\label{matrix-C}
C=\begin{pmatrix}
    c_1 & \Delta_1 & 0 & \dots & 0   & \Delta_2 \\
    \Delta_3 & c_2 & \Delta_2 & \dots & 0   & 0 \\
    \dots & \dots & \dots & \dots & \dots & \dots \\
    \Delta_n & 0 & 0 & \dots   & \Delta_1 & c_n
\end{pmatrix}
\end{equation}

\begin{prop}
Let $\sigma\subset T_z \R^{2n}$ be a 2-subspace such that
$\theta_i|_\sigma=0$ for all $i$. Then $\sigma$ is an integral element if and only if there exist real parameters
$c_i, 1\le i\le n$,
such that $\rank C=n-2$, and $C\omega|_\sigma=0$.
Here $\omega$ is the vector with components $\omega_i$.
\end{prop}
\begin{proof}
Suppose $\sigma$ is an integral element.
Then for every $i$ the forms $\omega_i$ and $\omega_{i+1}$ are linearly independent on $\sigma$. Otherwise, if for some $i$
they are linearly dependent, then by \eqref{equation-on-integral-element}
all the forms $\omega_i$ will be multiples of one another
on $\sigma$, which is absurd because $\dim\sigma=2$.
Since every three forms $\omega_i$ are linearly dependent on $\sigma$, for some $a_i, b_i$
we have
\begin{equation}\label{omega-a-b}
\omega_{i+2}=a_i\omega_i+b_{i+1}\omega_{i+1}.
\end{equation}
Wedge-multiplying both parts by  $\omega_{i+1}$ and using
\eqref{equation-on-integral-element} we obtain
$a_i=-\Delta_{i+1}^{-1}\Delta_{i+2}$. We put
$c_i=-b_i\Delta_i$. Then \eqref{omega-a-b} turns into
\begin{equation}\label{equation-on-integral-element-i}
\Delta_{i+2}\omega_i+c_{i+1}\omega_{i+1}+\Delta_{i+1}\omega_{i+2}=0.
\end{equation}
We put these scalar equations in a matrix-vector form
$C\omega=0$. Note that $\rank C\ge n-2$ because $C$ has a nonzero minor of order $n-2$ in the left lower corner.
Since $\dim\sigma=2$, we have exactly $\rank C= n-2$,
hence the conclusion. Clearly, the converse is also true.
\end{proof}

For every polygon with vertices $z_i, 1\le i\le n$,
we introduce a special integral element. We put
\[
d_i=\Delta_{i-1,i+1}=\det(r_{i-1},r_{i+1}).
\]
\begin{prop}\label{prop-special-int-elem}
Putting $c_i=-d_i$ defines an integral element.
\end{prop}
\begin{proof}
We claim that if $c_i=-d_i$, then $\rank C=n-2$.
To see that, we check that $Cr=0$, where $r$ is a formal
vector with components $r_i$. Since the vectors $r_i$
are not all collinear, we will have two linearly independent
vectors in the null-space of $C$, hence $\rank C=n-2$.

We check that
\begin{equation}\label{check-special-int-elem}
\Delta_{i+2}r_i-d_{i+1}r_{i+1}+\Delta_{i+1}r_{i+2}=0.
\end{equation}
Without loss of generality, for simplicity of notation,
we put $i=0$. The equation \eqref{check-special-int-elem}
turns into
\[
f(r_0,r_1,r_2):= r_0\det(r_1,r_2)+r_1\det(r_2,r_0)+r_2\det(r_0,r_1)=0.
\]
We note that $f$ is multilinear and alternating.
Since $r_0,r_1,r_2\in\R^2$ are linearly dependent, we have $f(r_0,r_1,r_2)=0$.
\end{proof}

\begin{remark}
  {\rm
  If $n$ is even, then putting $c_i=d_i$ also defines an integral element. In this case the vector with components
  $(-1)^i r_i$ is in the null-space of $C$.
  }
\end{remark}

\section{Convexity}

We determine what integral elements correspond to
convex billiards.

Let $\<. \, ,.\>$ denote the usual inner product in $\R^2$.
We put $s_i=|r_i|$.

\begin{lemma}\label{det-and-inner-product}
$\Delta_2\<r_0,r_1\>+\Delta_1\<r_1,r_2\>=s_1^2 d_1$.
\end{lemma}
\begin{proof}
Immediate by inner multiplying \eqref{check-special-int-elem}
by $r_{i+1}$.
\end{proof}

Let $\kappa_i$ be the curvature of $\Gamma$ at the midpoint
$\bar r_i$. We assume $\kappa_i>0$ because the tangent lines
at the points with zero curvature form a nowhere dense set.
We put $\kappa_i=\infty$ if $\Gamma$ has a corner at
$\bar r_i$.
\begin{prop}\label{prop-formula-for-ci}
Let $\{c_i\}$ define an integral element arising from a
convex billiard curve $\Gamma$. Then we have
\[
c_i=d_i-\frac{2\Delta_i\Delta_{i+1}}{\kappa_i s_i^3}.
\]
In particular, $c_i\le d_i$, where the equality occurs if
$\Gamma$ has a corner at $\bar r_i$.
\end{prop}
\begin{proof}
Let $\{z_i(t)\}$ be a curve in $M$ defined by \eqref{integral-manifold} through the point where $t=0$.
Then the quantities $r_i, \bar r_i$, etc., will depend on $t$. For simplicity of notation, we derive the needed formula for $i=1$. By a prime we denote the derivative with respect
to $t$ at $t=0$.
We first assume $\kappa_1<\infty$ and $\bar r'_1\ne 0$.
By the definition of the outer billiard, we put
$r_1=\lambda \bar r'_1$. Then
$r'_1=\lambda'\bar r'_1+ \lambda\bar r''_1$.

Since $\bar r_1 \in \Gamma$, the normal component of the acceleration $\bar r''_1$ has the form
$a_N=\kappa_1 v^2$, here $v=|\bar r'_1|$.
The inner unit normal vector to $\Gamma$ at $\bar r_1$ has the form $N_1=-s_1^{-1}Jr_1$, here $J$ is the counterclockwise rotation by $\pi/2$. Then we have
\[
\omega_1(z')=\det(r'_1,r_1)
=\det(\lambda'\bar r'_1+ \lambda\bar r''_1,r_1)
=\det(\lambda\kappa v^2 N_1,\lambda \bar r'_1)
=\frac{\lambda^2 v^2 \kappa_1}{s_1}\det(-Jr_1,\bar r'_1).
\]
Note $\lambda^2 v^2=s_1^2$, and $\det(-Jx,y)=\<x,y\>$.
Then
$\omega_1(z')=\kappa_1 s_1 \<\bar r'_1,r_1\>$, and
on the integral element we have
\begin{equation}\label{formula-for-ci-1}
\omega_1=\kappa_1 s_1 \<d\bar r_1,r_1\>.
\end{equation}
We have obtained \eqref{formula-for-ci-1} assuming that
$\bar r'_1\ne0$. However, if $\bar r'_1=0$, then \eqref{formula-for-ci-1}
still holds because in this case $\omega_1(z')=0$.
Using \eqref{dz-i}, we write
\[
2d\bar r_1
=dz_1+dz_2
=\Delta_1^{-1}(\omega_1 r_0+\omega_0 r_1)
+\Delta_2^{-1}(\omega_2 r_1+\omega_1 r_2),
\]
\[
2\Delta_1\Delta_2\<d\bar r_1,r_1\>
=\omega_1(\Delta_2\<r_0,r_1\>+\Delta_1\<r_1,r_2\>)
+(\Delta_2\omega_0+\Delta_1\omega_2)s_1^2.
\]
Using Lemma \ref{det-and-inner-product}
we get
\begin{equation}\label{formula-for-ci-2}
2s_1^{-2}\Delta_1\Delta_2\<d\bar r_1,r_1\>
=d_1 \omega_1+\Delta_2\omega_0+\Delta_1\omega_2,
\end{equation}
and using
\eqref{formula-for-ci-1} we get
\begin{equation}\label{formula-for-ci-3}
\frac{2\Delta_1\Delta_2}{\kappa_1 s_1^3}\omega_1
=d_1 \omega_1+\Delta_2\omega_0+\Delta_1\omega_2.
\end{equation}
We have obtained \eqref{formula-for-ci-3} assuming
$\kappa_1<\infty$. However, if $\kappa_1=\infty$,
then \eqref{formula-for-ci-3}
still holds because in this case $d\bar r_1=0$, and
\eqref{formula-for-ci-2} implies \eqref{formula-for-ci-3}.
Now by comparing \eqref{formula-for-ci-3}
with \eqref{equation-on-integral-element-i}
we obtain the desired equation for $c_1$.
\end{proof}

We call an integral element {\em convex} if $c_i\le d_i$.
Our results on the empty interior of the set of periodic
orbits follow from the absence of convex integral elements.

\section{The cases $n=3$ and $n=4$}

We first consider the simple cases $n=3$ and $n=4$.

\begin{thm}
For every triangle $\{z_1,z_2,z_3\}$, there is only one
integral element $c_i=\Delta$, here $\Delta$ is one-half
the area of the triangle. This integral element is not
convex. Hence, the set of period 3 orbits has empty interior.
\end{thm}
\begin{proof}
Note $\Delta_i=\Delta>0$.
Likewise, say $d_1=\Delta_{02}=-\Delta<0$.
Since $\rank C=1$, we have $c_i=\Delta>0$.
Then the convexity means that $0<c_1\le d_1<0$,
which is absurd.
\end{proof}

\begin{thm}
For every convex quadrilateral $\{z_i\}$, there is only one convex integral element $c_i=d_i$, that is, the
midpoints $\bar r_i$ are the corners of $\Gamma$.
Hence, the set of period 4 orbits has empty interior
unless there are 4 corners of $\Gamma$ forming
a parallelogram.
\end{thm}
\begin{proof}
Let $\{c_i\}$ define a convex integral element.
Then $\rank C=2$. In particular the 3-minor in the rows
1,2,3 and columns 1,3,4 vanishes. This condition yields
$c_1+c_3=0$. By Proposition \ref{prop-special-int-elem},
the numbers $\{-d_i\}$ always define an integral element,
hence $d_1+d_3=0$. By the convexity condition,
$c_i\le d_i$. Then $0=c_1+c_3 \le d_1+d_3=0$.
Then we must have $c_i=d_i$ for $i=1,3$.
Similarly, it holds for the remaining values $i=2,4$.
\end{proof}

For completeness, we note that for $n=4$,
all integral elements are defined by the equations
\[
c_1+c_3=0,\qquad
c_2+c_4=0,\qquad
c_1 c_2+\Delta_2\Delta_4=\Delta_1\Delta_3.
\]
They form a smooth curve, a hyperbola, unless
$\Delta_2\Delta_4=\Delta_1\Delta_3$, which means
the quadrilateral is a trapezoid.

\section{The case $n=5$}

We describe integral elements for $n=5$.
\begin{prop}\label{prop-int-elem-n5}
For $n=5$, integral elements are defined by the equation
\begin{equation}\label{equation-int-elem-n5}
c_1 c_2=c_4 \Delta_2+\Delta_1\Delta_3
\end{equation}
and the equations obtained from \eqref{equation-int-elem-n5}
by shifting indices.
\end{prop}
\begin{proof}
The first three columns of the matrix $C$ given by
\eqref{matrix-C} are linearly independent.
Then $\rank C=n-2=3$
if and only if the columns 4 and 5 are linear combinations
of columns 1, 2, and 3.
We include calculations for column 4 and leave the rest to
the reader. Denote the coefficients of the linear combination by $a_1, a_2, a_3$. Then we have
\[
a_1 c_1+a_2\Delta_1=0,\quad
a_1\Delta_3+a_2 c_2 +a_3\Delta_2=0,\quad
a_2\Delta_4+a_3 c_3=\Delta_3,\quad
a_3\Delta_5=c_4,\quad
a_1\Delta_5=\Delta_1.
\]
We find
$
a_1=\Delta_1/\Delta_5,\;
a_3=c_4/\Delta_5,\;
a_2=-c_1/\Delta_5
$.
Plugging them in the remaining equations, we obtain
\[
\Delta_1\Delta_3-c_1c_2+c_4\Delta_2=0,\qquad
\Delta_3\Delta_5-c_3c_4+c_1\Delta_4=0
\]
as desired.
\end{proof}

For a periodic orbit in outer billiard, we can introduce a winding number if we regard the polygon as a path.
Denote the interior angles by $\alpha_i$. Then
the angle from $r_{i-1}$ to $r_i$, that is,
the exterior angle at $z_i$ will be $\delta_i=\pi-\alpha_i$.
Then the winding number will be $m=(2\pi)^{-1}\sum \delta_i$.
A $(n,m)$-{\it orbit} is a period $n$ orbit with winding
number $m$. We have $\sum\alpha_i=\pi(n-2m)$, hence
$0<2m<n$. In particular, for $n=5$, there may be
$(5,1)$ and $(5,2)$-orbits.

Note that the angle from $r_{i-1}$ to $r_{i+1}$ is equal to $\delta_i+\delta_{i+1}=2\pi-\alpha_i-\alpha_{i+1}$.
Recall the notation $s_i=|r_i|$. Then we have
\begin{equation}\label{di}
d_i=\det(r_{i-1},r_{i+1})
=-s_{i-1} s_{i+1}\sin(\alpha_i+\alpha_{i+1}).
\end{equation}

\begin{thm}
The set of $(5,2)$-orbits has empty interior.
\end{thm}
\begin{proof}
We show that for $(5,2)$-orbits, there are no convex
integral elements. Let $\{z_i\}$ be a $(5,2)$-orbit,
and let $\{c_i\}$ define a convex integral element.
By Proposition \ref{prop-special-int-elem},
the numbers $\{-d_i\}$ define an integral element,
hence by \eqref{equation-int-elem-n5} we obtain
\begin{equation*}
d_1d_2=-d_4\Delta_2+\Delta_1\Delta_3.
\end{equation*}
Subtracting the later from \eqref{equation-int-elem-n5} we obtain
\begin{equation}\label{c1c2-d1d2}
c_1c_2-d_1d_2=(c_4+d_4)\Delta_2.
\end{equation}
Since $\sum\alpha_i=\pi(n-2m)=\pi$, we have
$\alpha_i+\alpha_{i+1}<\pi$. Then by \eqref{di}, $d_i<0$.
By convexity,
\[
c_i\le d_i<0,\qquad
c_1c_2-d_1d_2\ge 0,\qquad
c_4+d_4<0,
\]
contradicting \eqref{c1c2-d1d2}.
\end{proof}

\begin{remark}
{\rm
For (5,1) orbits our method does not work because
convex integral elements do exist.}
\end{remark}

\section{The case $n=6$}

For even $n$, there are open sets of period $n$ orbits
in which all the midpoints $\bar r_i$ are corners of $\Gamma$.
For instance, if $\Gamma$ is a triangle, then there is
an open set of period 6 orbits.

We first describe integral elements for $n=6$.
\begin{prop}\label{prop-int-elem-n6}
For $n=6$, integral elements are defined by the equations
\begin{equation}\label{equation-int-elem-n6}
\Delta_5(c_1c_2-\Delta_1\Delta_3)
+\Delta_2(c_4c_5-\Delta_4\Delta_6)=0,\qquad
\Delta_5(c_1\Delta_4-c_5\Delta_3)
+c_3(c_4c_5-\Delta_4\Delta_6)=0,
\end{equation}
and the equations obtained from \eqref{equation-int-elem-n6}
by shifting indices.
\end{prop}
The proof is similar to that of Proposition \ref{prop-int-elem-n5}, and we leave it to the reader.

For $n=6$, there may be (6,1) and (6,2)-orbits.
We call a (6,2)-orbit {\it paradoxical} if
$\alpha_{i-1}+\alpha_i>\pi$ and $\alpha_{i+1}+\alpha_i>\pi$
for some $i$. The author does not know whether paradoxical
orbits exist for a convex curve $\Gamma$.
\begin{thm}

The set of non-paradoxical (6,2)-orbits has empty interior
unless there is an orbit $\{z_i\}$ so that the midpoints
$\bar r_i$ are the corners of $\Gamma$.
\end{thm}
\begin{proof}
We show that for non-paradoxical $(6,2)$-orbits,
there are no convex integral elements except when all
the midpoints $\bar r_i$ are the corners of $\Gamma$.
Let $\{z_i\}$ be such an orbit,
and let $\{c_i\}$ define a convex integral element.
We use only the first equation in \eqref{equation-int-elem-n6}, which we rewrite in the form
\begin{equation}\label{equation-int-elem-n6-rewritten}
\Delta_5c_1c_2+\Delta_2c_4c_5
=\Delta_1\Delta_3\Delta_5+\Delta_2\Delta_4\Delta_6.
\end{equation}
By Proposition \ref{prop-special-int-elem},
the numbers $\{-d_i\}$ also define an integral element,
hence by \eqref{equation-int-elem-n6-rewritten} we obtain
\begin{equation*}
\Delta_5d_1d_2+\Delta_2d_4d_5
=\Delta_1\Delta_3\Delta_5+\Delta_2\Delta_4\Delta_6.
\end{equation*}
Subtracting the later from \eqref{equation-int-elem-n6-rewritten} we obtain
\begin{equation}\label{c1c2-d1d2-n6}
\Delta_5(c_1c_2-d_1d_2)
+\Delta_2(c_4c_5-d_4d_5)=0.
\end{equation}
Note that for a (6,2)-orbit, $\sum\alpha_i=2\pi$.
By \eqref{di}, $d_i>0$ exactly when
$\alpha_i+\alpha_{i+1}>\pi$.
If two of the numbers $d_i$ are positive, then they
must be consecutive. Since the orbit is non-paradoxical,
they cannot be consecutive either. Hence at most one
of these numbers may be positive. For definiteness,
let it be $d_6$.

By convexity, for all $i\ne 6$,
$c_i\le d_i\le 0$, hence
$c_1c_2-d_1d_2\ge0$ and $c_4c_5-d_4d_5\ge0$.
If $\Gamma$ is smooth, then these inequalities are strict,
contradicting \eqref{c1c2-d1d2-n6}, and we stop here.

In the general case, we have
\[
c_1c_2=d_1d_2, \qquad
c_4c_5=d_4d_5.
\]
We first consider the case in which indeed $d_6>0$.
Then since $\sum\alpha_i=2\pi$, we have $d_2<0$ and $d_4<0$.
Then we have $c_i=d_i$ for $i=1,2,4,5$. We show that it also
holds for $i=3,6$.
By shifting indices in \eqref{equation-int-elem-n6-rewritten},
we obtain
\begin{equation*}\label{equation-int-elem-n6-shifted}
\Delta_6c_2c_3+\Delta_3c_5c_6
=\Delta_1c_3c_4+\Delta_4c_6c_1
=\Delta_1\Delta_3\Delta_5+\Delta_2\Delta_4\Delta_6.
\end{equation*}
Since these equations hold for $d_i$ in place of $c_i$,
and already $c_i=d_i$ for $i=1,2,4,5$, we obtain
\[
\Delta_6d_2(c_3-d_3)+\Delta_3d_5(c_6-d_6)
=\Delta_1d_4(c_3-d_3)+\Delta_4d_1(c_6-d_6)=0,
\]
in which all terms are non-negative.
Since $d_2<0$, we have $c_3=d_3$. If $c_6\ne d_6$,
then $d_1=d_5=0$, which contradicts $\sum\alpha_i=2\pi$.
Hence $c_i=d_i$ for all $i$, as desired.

Finally, we consider the case in which all $d_i\le0$,
including $d_6$. Then arguing as above, we have
$c_ic_{i+1}=d_id_{i+1}$ for all $i$.
If some $d_i=0$, then there are at most two of them,
and they must be consecutive.
Then again $c_i=d_i$ for all $i$, as desired.
\end{proof}

\begin{remark}
{\rm
For (6,1) orbits our method does not work because
there exist convex integral elements with $c_i\ne d_i$.
}
\end{remark}

\end{document}